\documentclass[10pt]{amsart}
\pdfoutput=1
\usepackage{amsmath}
\usepackage{amsfonts}
\usepackage{amssymb}
\usepackage{amsthm, enumerate, enumitem}
\usepackage{hyperref}
\usepackage{comment, enumitem}
\usepackage{float}
\usepackage{youngtab}
\usepackage{ytableau}
\usepackage{rotating}
\usepackage{subcaption}
\usepackage{xcolor}
\usepackage{bbding}
\usepackage{tikz}
\usepackage{tikz-cd}
\usepackage[paper=a4paper, margin=3.5cm]{geometry}

\def\NN{{\NZQ N}}

\def\RR{{\NZQ R}}
\def\CC{{\NZQ C}}

\def\PP{{\NZQ P}}

\def\frk{\frak}               

\def\Phi{{\frk n}}
\def\Phi{{\frk N}}
\def\MI{{\mathcal I}}

\def\opn#1#2{\def#1{\operatorname{#2}}} 
\opn\chara{char} \opn\length{\ell} \opn\pd{pd} \opn\rk{rk}
\opn\projdim{proj\,dim} \opn\injdim{inj\,dim} \opn\rank{rank}
\opn\spn{span}\opn\Seg{Seg}
\opn\depth{depth} \opn\grade{grade} \opn\height{height}
\opn\embdim{emb\,dim} \opn\codim{codim}

\opn\Tr{Tr} \opn\bigrank{big\,rank}
\opn\superheight{superheight}\opn\lcm{lcm}
\opn\trdeg{tr\,deg}
\opn\reg{reg} \opn\lreg{lreg} \opn\ini{in} \opn\lpd{lpd}
\opn\size{size}\opn\bigsize{bigsize}
\opn\cosize{cosize}\opn\bigcosize{bigcosize}
\opn\sdepth{sdepth}\opn\sreg{sreg}
\opn\link{link}\opn\fdepth{fdepth} \opn\trdeg{trdeg} \opn\mod{mod}
\opn\spann{span}
\opn\div{div} \opn\Div{Div} \opn\cl{cl} \opn\Cl{Cl}
\opn\Spec{Spec} \opn\Supp{Supp} \opn\supp{supp} \opn\Sing{Sing}
\opn\Ass{Ass} \opn\Min{Min}\opn\Mon{Mon} \opn\dstab{dstab} \opn\astab{astab}
\opn\Syz{Syz}

\opn\Ann{Ann} \opn\Rad{Rad} \opn\Soc{Soc} \opn\Aut{Aut}
\opn\Im{Im} \opn\Ker{Ker} \opn\Coker{Coker} \opn\Am{Am}
\opn\Hom{Hom} \opn\Tor{Tor} \opn\Ext{Ext} \opn\End{End}
\opn\Aut{Aut} \opn\id{id}

\opn\nat{nat}
\opn\pff{pf}
\opn\Pf{Pf} \opn\GL{GL} \opn\SL{SL} \opn\mod{mod} \opn\ord{ord}
\opn\Gin{Gin} \opn\Hilb{Hilb}\opn\sort{sort}
\opn\S{S} \opn\dim{dim} \opn\supp{supp}\opn\trdeg{trdeg}\opn\sort{sort}
\opn\aff{aff} \opn\con{conv} \opn\relint{relint} \opn\st{st}
\opn\lk{lk} \opn\cn{cn} \opn\core{core} \opn\vol{vol}
\opn\link{link} \opn\star{star}\opn\lex{lex}
\opn\conv{conv} \opn\Ehr{Ehr}\opn\Pic{Pic}
\opn\Conv{Conv}
\opn\ML{ML-deg}\opn\Id{Id}
\opn\gr{gr}

\def\pot#1#2{#1[\kern-0.28ex[#2]\kern-0.28ex]}

\opn\dirlim{\underrightarrow{\lim}}
\opn\inivlim{\underleftarrow{\lim}}
\def\Implies{\ifmmode\Longrightarrow \else
        \unskip${}\Longrightarrow{}$\ignorespaces\fi}
\def\implies{\ifmmode\Rightarrow \else
        \unskip${}\Rightarrow{}$\ignorespaces\fi}
\def\iff{\ifmmode\Longleftrightarrow \else
        \unskip${}\Longleftrightarrow{}$\ignorespaces\fi}

\let\:=\colon

\newtheorem{Theorem}{Theorem}[section]
 \newtheorem{Lemma}[Theorem]{Lemma}
 \newtheorem{Corollary}[Theorem]{Corollary}
 \newtheorem{Proposition}[Theorem]{Proposition}
 \newtheorem{Remark}[Theorem]{Remark}

 \newtheorem{Example}[Theorem]{Example}
 
 \newtheorem{Definition}[Theorem]{Definition}

 \newtheorem{Question}[Theorem]{Question}

\DeclareMathOperator{\tr}{tr}

\DeclareMathOperator{\diag}{diag}

\def\CC{\mathbb{C}}

\def\RR{\mathbb{R}}

\def\PP{\mathbb{P}}
\def\NN{\mathbb{N}}

\title{Self-inverse linear subspaces of matrices}
 
 \author{Rodica Andreea Dinu}
\address{%
	 Simion Stoilow Institute of Mathematics of the Romanian Academy, Calea Grivitei 21, 010702, Bucharest, Romania}
	\email{rdinu@imar.ro}

\author{Martin Vodi\v{c}ka}
\address{Šafárik University, Faculty of Science, Jesenná 5, 04154 Košice, Slovakia}
	\email{martin.vodicka@upjs.sk}

\begin{document}

\maketitle

\begin{abstract}
    We study linear subspaces of matrices whose inverse spaces are also linear. Based on the fact that any linear space containing the identity matrix and whose inverse space is linear must be a self-inverse space, we introduce such spaces as self-inverse spaces. In fact, as we will show, self-inverse spaces are finite-dimensional complex unitary Jordan algebras. We provide an algebraic classification of all self-inverse spaces of small type and a classification of small-dimensional self-inverse spaces up to isomorphism.
\end{abstract}

\section{Introduction}
The structural study of matrix subspaces is a cornerstone of modern linear algebra, with important implications for algebraic geometry, quantum information and algebraic statistics.

While a matrix $A$ is classically known as involutory if $A=A^{-1}$, this paper shifts the focus toward a more global structural property: the characterization and classification of linear subspaces of matrices $L$ such that the set of inverses of all regular matrices within $L$, denoted $L^{-1}$, also forms a linear space.

Our investigation establishes a fundamental connection: any linear space $L$ containing the identity matrix and satisfying the condition that $L^{-1}$ is a linear space must be a self-inverse subspace, i.e., $L=L^{-1}$. Based on this fact, we call such spaces \textit{self-inverse} spaces. We establish key structural and algebraic properties of these subspaces, demonstrating that they are closed under matrix powers and a symmetric product operation. As we will show, self-inverse spaces are precisely Jordan subalgebras of the Jordan algebra $\CC^{n\times n}$ which contains an identity. We also refer the reader to \cite{jensen} and \cite{jordan}.
We proved that there is a close relationship between self-inverse spaces and the root multiplicities of the minimal polynomials of their elements. By using this fact, we provide a classification of all self-inverse spaces of small type and a classification of all small-dimensional self-inverse subspaces up to an isomorphism. 

The motivation for this study is deeply rooted in \cite{3264}, where Bernd Sturmfels poses critical problems regarding the degree of the inversion map and the geometry of spaces where the inverse space is also linear.

In algebraic statistics, a central problem in Gaussian graphical models is the relationship between the covariance matrix $\Sigma$ and the concentration matrix $K=\Sigma^{-1}$, see, for e.g. \cite{jensen} and \cite{STZ}. Statistical constraints often require $K$ to lie within a specific linear subspace $L$ defined by the absence of edges in a graph, see \cite{lauritzen} and \cite{jensen}. 
A self-inverse subspace represents a symmetric statistical scenario where both the covariance and concentration matrices belong to the same linear space (i.e., $K\in L \Rightarrow K^{-1}\in L$). Such models provide scenarios where the Maximum Likelihood Estimation (MLE) exhibits unique algebraic properties, such as a rational ML estimator or a specific ML degree.  

By classifying self-inverse spaces, we are essentially finding the building blocks of these statistical models. For instance, the space of symmetric $2\times 2$ matrices $(S_2)$ and the space of upper-triangular $2\times 2$ matrices $(U_2)$ identified in our dimension 3 classification, correspond to specific constraints in statistical interactions. Understanding these subspaces allows for a clearer view of how internal matrix symmetries (captured by the minimal polynomial) translate into global model constraints.

When the matrices are symmetric, a study of linear spaces whose inverse space is a linear space was done by Bik, Henrik, and Sturmfels in \cite{jordan}. A structure result for Euclidean Jordan algebras (so, over $\RR$) appeared already in 1934, due to Jordan, von Neumann, and Wigner in \cite{jordan_mechanics}. Results that generalize their work are due to Zel'manov, in the case of simple and prime Jordan algebras, but in infinite dimensions \cite{zelmanov1, zelmanov2, zelmanov4}.
In addition, Jordan algebras are related to matroids: they can be viewed as a nonabelian generalization of partition matroids, i.e., direct sums of uniform matroids, \cite[Remark 3.7]{jordan}.

This paper is organized as follows. In Section~\ref{foundation}, we establish the foundational lemmas, proving that any linear space containing the identity with a linear inverse space is closed under all non-negative matrix powers, ultimately showing that such spaces are self-inverse. In fact, some of these properties appear in a different shape in the literature; we refer here to \cite{jensen} and, for symmetric matrices, \cite{jordan}. However, our approach is different, and we present them for the completeness of the article.

In Section~\ref{jordan}, we study several properties of self-inverse spaces of matrices in relation to Jordan algebras. It is already known a decomposition result for Jordan algebras, namely the so-called {\em Peirce decomposition}, \cite{taste}. However, the uniqueness of the decomposition is not that clear in the literature and we prove in Theorem~\ref{uniqueness}, the fact that there is a unique way to write a Jordan algebra as a product of its irreducible Jordan subalgebras.

In Section~\ref{sectype} we extend the concept of minimal polynomial for matrices to self-inverse spaces and, implicitly, to Jordan algebras. In Theorem~\ref{generated-by-single-matrix}, we prove that the spaces generated by two matrices are isomorphic if and only if their minimal polynomials have the same number of roots with the same multiplicities, which shows that the minimal polynomial of a matrix plays an important role in understanding the self-inverse spaces of matrices. Based on this fact, we introduce an algebraic variety defined by fixed minimal polynomial types.
In Section~\ref{classification_type}, we fix the type of the Jordan algebra, hence the type of minimal polynomial of its generic element. We give a classification of self-inverse spaces of small type.
Section~\ref{classification_dim} provides a comprehensive algebraic classification of low-dimensional self-inverses spaces. Specifically, we identify exactly one unique space in dimension one, two spaces in dimension two (in Theorem~\ref{dim2}), six spaces in dimension three (in Theorem~\ref{dim3}), and sixteen spaces in dimension four (in Theorem~\ref{dim4}). Finally, we propose several open problems regarding higher-dimensional classifications, the finiteness of these spaces, and their orbits under the action of the general linear group.

\section{Linear spaces of matrices with linear inverse space}\label{foundation}

We start the study of linear spaces of matrices with linear inverse spaces by proving several foundational lemmas. They will lead to a generalization of \cite[Theorem 1.1]{jordan} corresponding to the symmetric case matrices.

\begin{Lemma}\label{polynomial-lemma}
Let $f\in\CC[x]$ be a polynomial of degree $k$ and let $R=\CC[x]/(f)$. Assume that $a_1,\dots,a_k\in\CC$ are different numbers that are not roots of $f$. Then $(x-a_1)^{-1},(x-a_2)^{-1},\dots,(x-a_k)^{-1}$ generate $R$ as a vector space. The inverses are considered in the ring $R$.
\end{Lemma}
\begin{proof}
It is sufficient to show that these elements are linearly independent. Assume $$\alpha_1(x-a_1)^{-1}+\alpha_2(x-a_2)^{-1}+\dots+\alpha_k(x-a_k)^{-1}=0$$
for some $\alpha_i\in\CC$.
We can multiply this equation by $(x-a_1)(x-a_2)\dots (x-a_k)$ and obtain $$\alpha_1(x-a_2)(x-a_3)\dots (x-a_k)+\dots +\alpha_k(x-a_1)(x-a_2)\dots(x-a_{k-1})=0.$$

Note that the polynomial on left-hand side is of degree $k-1$, therefore this equation holds also in $\CC[x]$. Thus we may plug in $x=a_i$ to obtain $\alpha_i=0$ which proves the linear independence.
\end{proof}

\begin{Lemma}\label{R<L}
Let $L$ be a linear space of matrices such that $\Id\in L$. Assume that $L^{-1}$ is a linear space. For any $A\in L$ and any $k\in \NN$, we have $A^k\in L$.
\end{Lemma}

\begin{proof}
    Let $R$ be the linear space generated by all non-negative powers of $A$. $R$ is closed under the multiplication of matrices, which means that $R$ is a ring. Moreover $R\cong \CC[x]/(f_A)$ where $f_A$ is the minimal polynomial of matrix $A$.

    The only non-invertible elements in $R$ are zero-divisors and this means that every regular matrix in $R$ has the inverse in $R$. Thus, $R^{-1}=R$.

    The intersection $L\cap R$ is a vector space and clearly $(L\cap R)^{-1} = L^{-1}\cap R^{-1}=L^{-1}\cap R$. Since $L^{-1}$ is a vector space also $L^{-1}\cap R$ is. Consider any $k$ complex numbers $a_1,\dots,a_k$ that are not the roots of $f_A$. We have $(A-a_i\Id)^{-1}\in (L\cap R)^{-1}$. By Lemma \ref{polynomial-lemma}, these matrices generate $R$ as a vector space and therefore $(L\cap R)^{-1}=R$.

    From this, it follows that $L\cap R=R$ and $R\subset L$, which proves the lemma.
\end{proof}

\begin{Lemma}\label{inverse}
 Let $L$ be a linear space of matrices such that $\Id\in L$. Assume that $L^{-1}$ is a linear space. Then $L=L^{-1}$.
\end{Lemma}

\begin{proof}
Let $A\in L$ be a regular matrix and let $R$ be the vector space generated by all non-negative powers of $A$. We know from Lemma \ref{R<L}, that $R\subset L$. Moreover, $A^{-1}\in R$, hence it follows that $A^{-1}\in L$. Thus, $L^{-1}\subset L$. Clearly, $\dim L=\dim L^{-1}$ since they are birational. Hence, we can conclude that $L^{-1}=L$.
\end{proof}
\begin{Theorem}\label{equiv}
Let $L$ be a linear space of matrices that contains at least one regular matrix. Then the following are equivalent:
\begin{enumerate}
    \item $L^{-1}$ is a linear space,
    \item $AL^{-1}A=L$ for some matrix $A\in L$,
    \item $AL^{-1}A=L$ for all regular matrices $A\in L$,
    \item $AB^{-1}A\in L$ for all regular matrices $A,B\in L$.
\end{enumerate}
\end{Theorem}

\begin{proof}
Clearly $(3)\Rightarrow (2)$ and $(4)\Rightarrow (3)$. Since multiplication by a regular matrix is a bijective linear map, it follows that $(2)\Rightarrow (1)$. We are left with the implication $(1)\Rightarrow (4)$.

Let $A$ be a regular matrix from $L$. Then the space $L':=A^{-1}L$ is a linear space. We have $(L')^{-1}=L^{-1}A$ which is a linear space. Moreover $\Id\in L'$. By Lemma~\ref{inverse}, we have $L'=(L')^{-1}$. Thus,
for any regular matrix $B\in L$ we have $$B^{-1}A=(A^{-1}B)^{-1}\in (L')^{-1}=L'=A^{-1}L.$$

Multiplying by $A$ from the left, we obtain  $AB^{-1}A\in L$, what we wanted to show.


\end{proof}

\begin{Definition}
We call a linear subspace of matrices \emph{self-inverse} if $L^{-1}=L$ and $\Id\in L$.
\end{Definition}

From the proof of Theorem \ref{equiv}, it can be seen that any linear space $L$ that satisfies the condition that the space $L^{-1}$ is linear can be written in the form $AL'$ where $L'$ is a self-inverse space. Thus, from now on, it makes sense to consider only self-inverse spaces.


\begin{Theorem}\label{equiv-cond-SIS}
Let $L$ be a linear subspace of matrices that contains the identity. Then the following are equivalent:
\begin{enumerate}
    \item $L$ is self-inverse,
    \item $A^k\in L$ for all $k\in\NN$, $A\in L$,
    \item $A^2\in L$ for all  $A\in L$,
    \item $AB+BA\in L$ for all $A,B\in L.$
\end{enumerate}
\end{Theorem}
\begin{proof}
Clearly $(2)\Rightarrow (3)$. To prove $(3)\Rightarrow (4)$, one notices that for $A,B\in L$ we have $(A+B)^2-A^2-B^2=AB+BA\in L$.

For $(4)\Rightarrow (2)$, one can proceed by mathematical induction. For $k=1$ the statement holds. Now assuming that $A^k\in L$ by substituting $B=A^k$ into condition $(4)$, we obtain $2A^{k+1}\in L$, proving the statement.

$(1)\Rightarrow (2)$ holds by Lemma \ref{R<L}. 

We will finish the proof with the implication $(2)\Rightarrow(1)$. 
Assume that $A\in L$ is regular. $A^{-1}$ is a polynomial in $A$, therefore $A^{-1}\in L$. Thus, $L^{-1}\subset L$, which implies that $L^{-1}=L$, since they both have the same dimension and both are irreducible.
\end{proof}
We point out here that, in Theorem~\ref{equiv-cond-SIS}, the equivalence of $(1)$ and $(4)$ may also be consulted in \cite[Lemma 1]{jensen}.

\section{Jordan algebras}\label{jordan}

In this section, we will make connections between self-inverse spaces and Jordan algebras.

\begin{Definition}
 A Jordan algebra over field $k$ is a vector space $J$ over a field $k$ equipped with a bilinear product $\star:J\times J\rightarrow J$ that satisfies

 $$X*Y=Y*X, ((X*X)*Y)*X=(X*X)*(X*Y)$$
 for all $X,Y\in J$.
\end{Definition}

It is a standard result that the set of all $n\times n$ matrices with the operation $X*Y=\frac 12(XY+YX)$ forms a Jordan algebra (see for e.g., \cite{taste, jensen}). In fact, this construction works for any associative algebra.

\begin{Theorem}
Self-inverse spaces are precisely Jordan subalgebras of the Jordan algebra $\CC^{n\times n}$, which contains an identity.
\end{Theorem}
\begin{proof}
 Follows from Theorem \ref{equiv-cond-SIS}. 
\end{proof}

Thus, when we speak about self-inverse spaces, we may instead consider them as Jordan algebras. For example, we can ask if two self-inverse spaces are isomorphic as Jordan algebras.

Self-inverse spaces are, therefore, finite-dimensional complex Jordan algebras with a unity element (the identity matrix). For a unitary Jordan algebra $J$, we will denote by $\Id_J$, its unity element.

The following lemma is trivial to check, but it will be important later.

\begin{Lemma}\label{conjugation}
For any regular matrix $P$ the conjugation $A\mapsto PAP^{-1}$ is an automorphism of the Jordan algebra $\CC^{n\times n}$.  
\end{Lemma}

The next result is known as Peirce decomposition (\cite[Part II, Chapter 8]{taste}):

\begin{Theorem}\label{FZ}(Peirce Decomposition)
 Let $L$ be a Jordan algebra. Let $I\in L$ have the property that $I*I=I$. Then the two vector subspaces:
\begin{itemize}
    \item $F_I:=\{A\in L: A*I=A\}$
    \item $Z_I:=\{A\in L: A*I=0\}$
\end{itemize}
are both unitary Jordan algebras. Moreover, we have $F_I*Z_I=0$. In addition, the space $F_I\oplus Z_I$ is a Jordan subalgebra of $L$.

Let $H_I:=\{A\in L: A*I=A/2\}$ be a vector subspace. Then $$L=F_I\oplus Z_I\oplus H_I.$$ Moreover, $H_I*H_I\subseteq F_I\oplus Z_I$ and $H_I*(F_I\oplus Z_I)\subseteq H_I$.
\end{Theorem}

\begin{proof}
Even though it is a well-known result, we will sketch the visualization of the proof in the case where $L$ is a subalgebra of $\CC^{n\times n}$.

Since $I^2=I$, we may assume that $I$ is a diagonal matrix with only 1 and 0 on the diagonal (otherwise, due to Lemma \ref{conjugation}, we can get to this case by conjugation). Let us assume that the first $m$ entries are equal to 1 and the rest $n-m$ entries are 0.

The set $F_I$ is formed by all matrices from $L$ that have non-zero entries in the left upper $m\times m$ corner. Clearly, this set is closed under the operation $*$. Moreover, we can see that $\Id_{F_I}=I$. 

Analogously, the set $Z_I$ is formed by all matrices from $L$ that have non-zero entries only in the right bottom $(n-m)\times (n-m)$ corner. This set is also closed under the operation $*$. It follows that $Z_I$ is a Jordan algebra with identity element $\Id-I$.

For the second part, we notice that the space $H_I$ is formed by all matrices from $L$ that have non-zero entries only in the top right $m\times(n-m)$ corner and bottom left $(n-m)\times m$ corner.

It is trivial to check that all the product formulas involving the sets $F-I, Z_I, H_I$ hold there by multiplying matrices in the block forms.

\end{proof}

\begin{Definition}
    We say that a Jordan algebra $L$ is irreducible if it is not isomorphic to a product of two non-trivial self-inverse subspaces. That is $L\cong S\times T$ implies $S=0$ or $T=0$. Otherwise, we call it reducible.
\end{Definition}
\begin{Remark}
Usually, in the literature, it is called a direct sum and not a product of two Jordan algebras. In the case of finite-dimensional Jordan algebras, it is clearly the same. However, sometimes we want to differentiate between the direct sum of underlying vector spaces and the direct sum of Jordan algebras. For this reason, we prefer to speak about the direct sum of vector spaces and about the product of Jordan algebras. 
\end{Remark}

Every Jordan algebra can be written as a product of irreducible Jordan algebras. The next statements will tell us when this is possible.


\begin{Theorem}\label{splitting}
Let $L$ be a Jordan algebra and $S,T$ be subalgebras such that $L=S\oplus T$ as a vector space. Then $\varphi: S\times T\rightarrow L$ given by $\varphi(A,B)=A+B$ is an isomorphism if and only if $A*B=0$ for all $A\in S, B\in T$.
\end{Theorem}
\begin{proof}
 $(\Leftarrow)$: We have $$\varphi((A_1,B_1))*\varphi((A_2,B_2))=(A_1+B_1)*(A_2+B_2)=A_1*A_2+A_1*B_2+B_1*A_2+B_1*B_2=$$
 $$A_1*A_2+B_1*B_2=\varphi(A_1*A_2,B_1*B_2)=\varphi((A_1,B_1)*(A_2,B_2)).$$


 $(\Rightarrow)$: For any $A\in S, B\in T$ we have $0=(0,B)*(A,0)=(0+B)*(A+0)=A*B$ which implies that $A*B=0$.
\end{proof}

\begin{Theorem}\label{splitting-by-I}
Let $L$ be a Jordan algebra and let $I\in L$ be an element with the property $I*I=I$.
Then $\varphi: F_I\times Z_I\rightarrow L$ given by $\varphi(A,B)=A+B$ is an isomorphism if and only if $F_I\oplus Z_I=L$.
\end{Theorem}

\begin{proof}
 $(\Leftarrow)$: Clearly $F_I\cap Z_I=0$. Thus, if the spaces $L$ and $F_I\times Z_I$ have the same dimension, we must have $L=F_I\oplus Z_I$.

 $(\Rightarrow)$: By Theorem \ref{FZ}, we have $A*B=0$ for any $A\in F_I, B\in Z_I$. Thus, by Theorem \ref{splitting}, the map $\varphi$ is an isomorphism.
\end{proof}

\begin{Theorem}\label{the-other-part}
Let $L$ be a unitary Jordan algebra. Suppose that $L\cong S\times T$ for two Jordan subalgebras $S$ and $T$. Then there exists an element $I\in L$, $I\notin\{\Id_L,0\}$ such that $I*I=I$ and $F_I\cong S, Z_I\cong T$. 
\end{Theorem}
\begin{proof}
Clearly, $S$ and $T$ must also be unitary. Consider an element $\overline{I}:=(\Id_S,0)\in S\times T$. It is easy to check that  $\overline I*\overline I=\overline I$ and $F_{\overline I}=S\times\{0\}$ and $Z_{\overline I}=\{0\} \times T$.

If $\varphi: S\times T\rightarrow L$ is an isomorphism, then clearly $I:=\varphi(\overline I)$ is the element that satisfies the condition.
\end{proof}

The following result is a consequence of Theorem~\ref{splitting-by-I} and Theorem~\ref{the-other-part}.

\begin{Corollary}\label{iso}
 Let $L$ be a Jordan algebra. Then $L$ is isomorphic to a product of two non-trivial Jordan subalgebras if and only if there exists an element $I\in L$ with the property $I*I=I$ and $H_I=0$.   
\end{Corollary}

The following result, we show that there is a unique way to write a Jordan algebra as a product of its irreducible Jordan subalgebras. 

\begin{Theorem}[Uniqueness of the product]\label{uniqueness}
Let $L$ be a unitary Jordan subalgebra. If $L\cong S_1\times\dots \times S_k\cong T_1\times \dots\times T_l$, where all $S_i, T_i$ are irreducible and non-trivial, then $k=l$. Moreover, let $\varphi_S: S_1\times\dots\times S_k\rightarrow L$ and  $\varphi_T: T_1\times\dots\times T_k\rightarrow L$ be two isomorphisms and let $L_i:=\varphi(0\times 0\times\dots \times S_i\times\dots\times 0)$
and $L'_i:=\varphi(0\times 0\times\dots \times T_i\times\dots\times 0)$. Then $\{L_1,\dots,L_k\}=\{L'_1, \dots , L'_k\}$. In particular, we have that the set of factors $\{S_1,\dots ,S_k\}$ is, up to an isomorphism and ordering, unique.  
\end{Theorem}

\begin{proof}



Consider the elements $I_i=\varphi_S(0,0,\dots,0,\Id_{S_i},0,\dots,0)$. Note that $F_{I_i}=L_i$.

Let $J\in L$ be an element with the property $J*J=J$. Let us decompose $J=J_1+\dots +J_k$ where $J_i\in L_i$. Then $J*J=J_1*J_1+\dots+J_k*J_k$. It follows that $J_i*J_i=J_i$. However, since the space $L_i$ is irreducible, by Corollary \ref{iso}, we must have $J_i=I_i$ or $J_i=0$. Thus, it follows that $J=I_{i_1}+I_{i_2}+\dots+I_{i_m}$. Then $F_J=L_{i_1}\oplus \dots\oplus L_{i_m}\cong S_{i_1}\times\dots\times S_{i_m}$. Thus, if $m\ge 2$, the space $F_J$ is reducible. 

Let $I'_i=\varphi_T(0,0,\dots,0,\Id_{T_i},0,\dots,0)$. Since $I'_i*I'_i=I'_i$ and $F_{I'_i}\cong T_i$ is irreducible, it follows that $I'_i=I_{j}$ for some index $j$. Clearly, all the elements $I'_i$ are different, thus we have a bijection between sets $\{I_1,\dots , I_k \}$ and $\{I'_1,\dots , I'_l\}$ which proves the theorem. 
\end{proof}

\section{Type of minimal polynomials}\label{sectype}

The concept of a minimal polynomial is natural for matrices. We can extend this concept into our self-inverse spaces and Jordan algebras.

We will start with a trivial result.
\begin{Lemma}
    The intersection of two self-inverse spaces is a self-inverse space.
\end{Lemma}

\begin{proof}
Using condition $(3)$ in Theorem \ref{equiv-cond-SIS} it is obvious.
\end{proof}

Therefore, for any subset $S$ of $n\times n$ matrices, we may define the smallest self-inverse space (Jordan subalgebra containing the identity) that contains $S$ and we denote it by $[S]$. Clearly, $[S]$ is contained in the subring generated by $S$, but it is not necessarily equal to it. For example, if we take $S$ to be the set of all symmetric matrices, then $[S]=S$. However, the set of symmetric matrices does not form a subring.

In the case $S=\{A\}$, we have $[A]=\{f(A): f\in \CC[x]\}$.

\begin{Theorem}\label{generated-by-single-matrix}
Let $A$ and $B$ be two matrices. Then $[A]$ and $[B]$ are isomorphic if and only if the minimal polynomials of $A$ and $B$ have the same number of roots with the same multiplicities.
\end{Theorem}

\begin{proof}
In the case that $A$ and $B$ have the same minimal polynomial, clearly the map $\varphi:[A]\rightarrow [B]$ such that $f(A)\mapsto f(B)$ for any polynomial $F$ is an isomorphism.

For the sequence $(k_1,\dots,k_m)$ we define the following self-inverse space, which will be denoted by $\mathcal L_{k_1,\dots,k_m}$:

\noindent$\left\{\begin{bmatrix}
    \begin{bmatrix}
        x_{1,1}& x_{1,2}&x_{1,3}&\dots&x_{1,k_1}\\
        0&x_{1,1}& x_{1,2}&\dots&x_{1,k_1-1}\\
        0&0&x_{1,1}&\dots&x_{1,k_1-2}\\
        \vdots&\vdots&\vdots&\ddots&\vdots\\
        0&0&0&\dots&x_{1,1}
    \end{bmatrix} &  \mathbf{0}& \mathbf{0}\\
     \mathbf{0}&\ddots& \mathbf{0}\\
     \mathbf{0}& \mathbf{0}& \begin{bmatrix}
        x_{m,1}& x_{m,2}&x_{m,3}&\dots&x_{m,k_m}\\
        0&x_{m,1}& x_{m,2}&\dots&x_{m,k_m-1}\\
        0&0&x_{m,1}&\dots&x_{m,k_m-2}\\
        \vdots&\vdots&\vdots&\ddots&\vdots\\
        0&0&0&\dots&x_{m,1}
    \end{bmatrix}
\end{bmatrix}: x_{i,j}\in\CC\right\}.$

It is easy to check that this space is self-inverse space. It has dimension $k_1+\dots+k_m$. Assume that the minimal polynomial $m(A)$ of $A$ has $m$ roots $y_1,\dots,y_m$ with multiplicities $k_1,\dots,k_m$. We will show that $[A]\cong\mathcal L_{k_1,\dots,k_m}$.
Consider the matrix $X$ which satisfies $x_{i,1}=y_1$, $x_{i,2}=1$ and the rest of its entries is 0. Clearly, the minimal polynomial of $X$ is $m(A)$ since it is the matrix in Jordan normal form. Thus $[X]$ is of dimension $k_1+\dots+k_m$ which implies that $[X]=\mathcal L_{k_1,\dots,k_m}$. Thus $[A]\cong[X]=\mathcal L_{k_1,\dots,k_m}$. Thus, we showed the first implication, that if $A$ and $B$ have the roots with the same multiplicities, their self-inverse spaces are isomorphic.

For the other implication it suffices to show that $\mathcal L_{k_1,\dots,k_m}\cong\mathcal L_{k'_1,\dots,k'_{m'}}$ implies $\{k_1,\dots,k_m\}=\{k'_1,\dots,k'_{m'}\}$, where the last equality is the equality of multisets.

Clearly, we have that $k_1+\dots+k_m=k'_1+\dots+k'_{m'}$, based on the dimension argument.

Note that the characteristic polynomial of any matrix in $\mathcal L_{k_1,\dots,k_m}$ is $(x-x_{1,1})^{k_1}\dots(x-x_{m,1})^{k_m}$. Thus, any matrix from this space that has a minimal polynomial of degree $k_1+\dots+k_m$ must have a minimal polynomial of this form. Moreover, $x_{i,1}$ elements must be all different, otherwise the minimal polynomial would be of smaller degree.

However, this shows that if $\mathcal L_{k_1,\dots,k_m}\cong\mathcal L_{k'_1,\dots,k'_{m'}}$, then the matrix with minimal polynomial with roots of multiplicities $k_1,\dots,k_m$ must be mapped to a matrix with minimal polynomial with roots of multiplicities $k'_1,\dots,k'_{m'}$. This proves that $\{k_1,\dots,k_m\}=\{k'_1,\dots,k'_{m'}\}$.
\end{proof}

\begin{Remark}
Note that $\mathcal L_{k_1,\dots,k_m}\cong \mathcal L_{k_1}\times\dots\times \mathcal L_{k_m}$.
\end{Remark}


We have seen that the minimal polynomial of a matrix plays an important role in understanding self-inverse spaces of matrices. More specifically, the essential data captured by a minimal polynomial is found in its degree and multiplicity of roots. Thus, for $k_1\ge\dots\ge k_m$, it makes sense to define the following:

$$\mathcal M_{k_1,\dots,k_m}^n=\overline{\{A\in \CC^{n\times n}: \text{ minimal polynomial of $A$ has $m$ different roots}}$$
$$\overline{\text{with multiplicities $k_1,\dots,k_m$}\}}.$$

It is easy to see that $\mathcal M_{k_1,\dots,k_m}^n$ is an algebraic variety since all conditions are of polynomial type.
Sometimes (mainly when we speak abstractly about self-inverse spaces), we will omit the upper index $n$. We assume that this index is taken appropriately, i.e., it is the same as the size of matrices in our self-inverse space.

We say that a matrix $A$ is {\em of type $(k_1,\dots,k_m)$} if its minimal polynomial has $m$ different roots of multiplicities $k_1,\dots,k_m$. Hence, a generic matrix from $\mathcal M_{k_1,\dots,k_m}$ is of type $(k_1,\dots,k_m)$.

We can extend this concept even into finite-dimensional unitary Jordan algebras. For an element $X\in J$, is it polynomial well-defined, so also every element has its minimal polynomial. By $J_{k_1,\dots,k_m}$ we denote the closure of the set of all elements in $J$ with the minimal polynomial of corresponding type. Again, it is an algebraic subvariety of the vector space $J$.

Since $J$ is a union of all such sets and a vector space is irreducible as an algebraic variety, it follows that there exist $k_1,\dots,k_m$ such that $J_{k_1,\dots,k_m}=J$. This means that a generic element in $J$ is of this type, which is a generic type of minimal polynomial for the Jordan algebra $J$.

\section{Classification of self-inverse spaces of small type}\label{classification_type}

In this section, we will classify all self-inverse spaces (or unitary Jordan algebras) up to dimension four. In the literature, there is a general classification, but only for simple Jordan algebras. We refer here to \cite{jordan_mechanics} and also to the monograph \cite{taste}.

Our approach will be that we fix the type of Jordan algebra, i.e., the type of minimal polynomial of its generic element. We begin by giving a classification of unitary Jordan algebras of types $(1), (2)$ and $(1,1)$ of arbitrary dimension.


\subsection{Self-inverse spaces of type $(1)$}

Clearly, in any self-inverse space, the only matrix of type $(1)$ is a multiple of the identity, the only self-inverse space of type $(1)$ is one-dimensional and it is $\mathcal L_1$.

\subsection{Self-inverse of type $(2)$}

Let $L$ be a self-inverse space for which the general matrix is of the type $(2)$. Consider the set $L_0$ of singular (therefore, nilpotent)  matrices in $L$. Since every matrix in $L$ has a minimal polynomial of degree (at most) two, we have $A^2=0$ for all $A\in L_0$. This means $\tr(A)=0$ for all $A\in L_0$. In fact, $A\in L_0\Leftrightarrow \tr(A)=0$. This shows that $L_0$ is a vector subspace of $L$ of codimension 1. Moreover, for any matrix $A\in L_0$, clearly $A^2\in L_0$, and since $A*B=((A+B)^2-A^2-B^2)/2$, it shows that $L_0$ is also closed on $*$-product.

Consider any two matrices $A,B\in L_0$. Since $0=(A+B)^2=2(A*B)$, it follows that $A*B=0$. This means that once we fix a dimension, the space $L_0$ is uniquely determined (up to an isomorphism).

The unique $n$-dimensional space of type $(2)$ is therefore

$$\Gamma_n:=\left\{\begin{pmatrix}
a&b_1&b_2&\dots&b_{n-1}\\
0&a&0&\dots&0\\
0&0&a&\dots&0\\
\vdots&\vdots&\vdots&\ddots&\vdots\\
0&0&0&\dots&a
\end{pmatrix}: a,b_1,\dots,b_{n-1}\in\CC \right\}.$$

\subsection{Self-inverse spaces of type $(1,1)$}

Since a general matrix in such a space generates space $\mathcal L_{1,1}$ we can choose a matrix $A\in L$ such that $A^2=A$.

By Theorem~\ref{FZ}, we can decompose space $L=F_A\oplus Z_A\oplus H_A$. 
Moreover, $F_A\oplus Z_A$ is a subspace that is, by Theorem \ref{splitting-by-I}, isomorphic to $F_A\times Z_A$. It follows that $F_A$ and $Z_A$ must be of type $(1)$, otherwise a general element in $F_A\oplus Z_A$ will have bigger type than $(1,1)$. It follows that $\dim (F_A)=\dim(Z_A)=1$.

Thus, it remains to describe the space $H_A$. Note that for any two matrices $C,D\in H_A$, we have $C*D\in F_A\oplus Z_A$. We will prove that $C*D$ is a multiple of the identity.

First, let us consider the case when $C$ and $D$ are linearly dependent. It is sufficient to consider $C=D$. In this case, $C^2$ is a linear combination of $\Id$ and $C$, since the minimal polynomial of $C$ is of degree two. Since $C^2\in F_A\oplus Z_A$, it follows that $C^2$ is a multiple of the identity.

Let us consider matrices $C,D\in H_A$, such that $C,D,\Id$ are independent. We will show $C*D\in \spann(\Id,C,D)$:

Clearly $(C+D)^2=C^2+D^2+2(C*D)\in \spann(\Id,C+D)\subset\spann(\Id,C,D)$. Since $C^2\in\spann(\Id,C), D^2\in\spann(\Id,D)$ we obtain that $C*D\in\spann(\Id,C,D)$. However, $\spann(\Id,C,D)\cap (F_A\oplus Z_A)=\spann(\Id)$. From this, we can conclude that $C*D$ is a multiple of the identity.

Thus, the operation $*$ on $H_A$ is defined by a symmetric bilinear form $q : H_A\times H_A\rightarrow \CC$ and putting $C*D=q(C,D)\cdot \Id$. 

By the standard theory on complex symmetric bilinear forms, we can find matrices $C_2,\dots C_k$ and $D_1,\dots D_l$ that span the set $H_A$ such that $C_i^2=\Id$, $D_i^2=0$ and $C_i*C_j=D_i*D_j=0$ for all $i\neq j$ and $C_i*D_j=0$ for all $i,j$.

It also works in the opposite direction. If we fix non-negative integers $k,l$ and define $H_A$ as a span of matrices $C_1,\dots, C_k$ and $D_1,\dots, D_l$ with the relations as above, we obtain a self-inverse space of the type $(1,1)$.

Note that if we put $C_1:=2A-\Id$, then we also have $C_1*C_i=C_1*D_j=0$ and $C_1^2=\Id$. It follows that the self-inverse space $L$ is generated by matrices $\Id,C_1,\dots,C_k, D_1,\dots D_l$, which satisfy the relations above.

This is a full description of all self-inverse spaces of type $(1,1)$, and the previous construction depends only on integers $k\ge 1,l\ge 0$. That means a choice of these two integers uniquely (up to an isomorphism) determines a self-inverse space.

We will (recursively) construct an example of such a self-inverse space for any $k,l$. We will start with the case $k=1:$
\begin{itemize}
    \item $C_1=\diag(1,1,\dots,1,-1,\dots,-1)$ with exactly $l$ ones, and $l$ minus ones.
    \item $$D_i=\begin{pmatrix}
        \mathbf{0}& \diag(0,\dots,0,1,0,\dots,0)\\
       \mathbf{0}& \mathbf{0} 
    \end{pmatrix},$$
where $1$ is on the $i$-th position.
\end{itemize}

Let as assume we have a construction of $n\times n$ matrices $C_i$, $D_j$ for $k$, we will construct such matrices $C'_i, D'_J$ of size $2n\times 2n$ for $k+1$:

\begin{itemize}
    \item $C'_{k+1}=\diag(1,1,\dots,1,-1,\dots,-1)$ with exactly $n$ ones, and $n$ minus ones.
\item $$D'_i=\begin{pmatrix}
        \mathbf{0}& D_i\\
       D_i& \mathbf{0} 
    \end{pmatrix},$$
    \item $$C'_i=\begin{pmatrix}
        \mathbf{0}& C_i\\
       C_i& \mathbf{0} 
    \end{pmatrix}.$$
\end{itemize}

It is straightforward to check that these matrices satisfy all the conditions. We will denote this space by $\mathcal C_{k,l}$.

\section{Classification of self-inverse spaces of small dimension}\label{classification_dim}

In this section, we will describe all self-inverse spaces of dimension at most four.

\subsection{Dimension 1}

Since we require that a self-inverse space contains the identity, there is only one one-dimensional self-inverse space, and that is the space generated by the identity.

\subsection{Dimension 2}

\begin{Theorem}\label{dim2}
    There exist (up to isomorphism) only two different two-dimensional self-inverse spaces: $\mathcal L_2$ and $\mathcal L_{1,1}$.
\end{Theorem}

\begin{proof}
Let $L$ be a 2-dimensional self-inverse space. Clearly $\Id\in L$. Consider any matrix $A\in L$ such that $A$ is not a multiple of the identity. Then $\{A, \Id\}$ forms a basis of $L$ as a vector space. Clearly $L=[A]$ and the minimal polynomial of $A$ is of degree 2. By Theorem \ref{generated-by-single-matrix} we see that $[A]$ is isomorphic to either $\mathcal L_2$ or $\mathcal L_{1,1}$. By the same result, these two spaces are non-isomorphic, so the proof is finished.
\end{proof}

\subsection{Dimension 3}

\begin{Theorem}\label{dim3}
    There exist (up to isomorphism) only six different three-dimensional self-inverse spaces:
    
    \begin{itemize}
        \item $\mathcal L_3$, $\mathcal L_{2,1}$ and $\mathcal L_{1,1,1}$.
        \item $\Gamma_3:=\left\{\begin{pmatrix}
            a&b&c\\
            0&a&0\\
            0&0&a
        \end{pmatrix}\:\ a,b,c\in\CC\right\}.$
    \item $U_2:=\left\{\begin{pmatrix}
            a&b\\
            0&c
        \end{pmatrix}\:\ a,b,c\in\CC\right\}$ i.e the space of upper-triangular $2\times2$ matrices.
    \item $S_2:=\left\{\begin{pmatrix}
            a&c\\
            c&b
        \end{pmatrix}\:\ a,b,c\in\CC\right\}$ i.e the space of symmetric $2\times2$ matrices.
    \end{itemize}
\end{Theorem}

\begin{proof}
Let $L$ be a 3-dimensional self-inverse space. If $L=[A]$ for some matrix $A$, then by Theorem \ref{generated-by-single-matrix}, $L$ is isomorphic to one of the spaces $\mathcal L_3$, $\mathcal L_{2,1}$ and $\mathcal L_{1,1,1}$.

Now we may assume that every matrix in $L$ (which is not a multiple of the identity) has a minimal polynomial of degree two. Thus, it is a space of type $(1,1)$ or $(2)$. By the results of the previous section, we can see that there are two spaces of type $(1,1)$, namely $\mathcal C_{1,1}$ and $\mathcal C_{2,0}$. However, one can see that $\mathcal C_{1,1}\cong U_2$ and  $\mathcal C_{2,0}\cong S_2$.

Also, there is only one three-dimensional space of type $(2)$, and that is $\Gamma_3$.

\end{proof}

\begin{Remark} Clearly, all of the spaces which are of a different type are non-isomorphic. One can directly see that the spaces $U_2$ and $S_2$ are non-isomorphic by looking at the matrices with minimal polynomial with a double root. In the case of upper triangular matrices, this set is given by the equation $a=c$, therefore, it forms a plane in this subspace.

On the other hand, in the space of symmetric matrices, this set is given by the equation $(a+b)^2-4ab+c^2=(a-b+ic)(a-b-ic)=0$, therefore it is formed by two planes. Thus, these two spaces are non-isomorphic.
\end{Remark}

\begin{Remark}
Theorem~\ref{dim3} is a generalization of \cite[Theorem 5.1]{jordan}, when the case of symmetric matrices is treated.
\end{Remark}

\subsection{Dimension 4}

\begin{Theorem}\label{dim4}
    There exist (up to isomorphism) only 16 different four-dimensional self-inverse spaces:
    
    \begin{itemize}
        \item $\mathcal L_4$, $\mathcal L_{3,1}$, $\mathcal L_{2,2}$, $\mathcal L_{2,1,1}$, $\mathcal L_{1,1,1,1},$
        \item $U_2\times \mathcal L_1$, $S_2\times \mathcal L_1$, $\Gamma_3\times \mathcal L_1,\CC^{2\times 2},$
        \item  $\Gamma_{3,1}:=\left\{\begin{pmatrix}
            a&b&c\\
            0&a&0\\
            0&0&d
        \end{pmatrix}\:\ a,b,c,d\in\CC\right\},$
        \item  $\Gamma_{2,2}:=\left\{\begin{pmatrix}
            a&c&d\\
            0&b&c\\
            0&0&a
        \end{pmatrix}\:\ a,b,c,d\in\CC\right\},$
        \item  $\Gamma_4:=\left\{\begin{pmatrix}
            a&b&c&d\\
            0&a&0&0\\
            0&0&a&0\\
            0&0&0&a
        \end{pmatrix}\:\ a,b,c,d\in\CC\right\},$
        \item  $\Delta_4:=\left\{\begin{pmatrix}
            a&b&c&d\\
            0&a&0&b\\
            0&0&a&c\\
            0&0&0&a
        \end{pmatrix}\:\ a,b,c,d\in\CC\right\},$
        \item $\Omega_4:=\left\{\begin{pmatrix}
            a&b&c&d\\
            0&a&b&0\\
            0&0&a&0\\
            0&0&0&a
        \end{pmatrix}\:\ a,b,c,d\in\CC\right\},$
        \item $\Psi_4:=\left\{\begin{pmatrix}
            a&b&0&d\\
            c&a&d&0\\
            0&0&a&-c\\
            0&0&-b&a
        \end{pmatrix}\:\ a,b,c,d\in\CC\right\},$
        \item ${\Omega}_{2,2}:=\left\{\begin{pmatrix}
            a&c&0&0\\
            0&b&0&0\\
            0&0&a&d\\
            0&0&0&b
        \end{pmatrix}\:\ a,b,c,d\in\CC\right\}.$
    \end{itemize}
\end{Theorem}

\begin{proof}
We will again proceed by casework, as in the lower dimensions. First, we will look at the type of a generic matrix in the self-inverse space $L$.

In the case where generic matrix $A$ has minimal polynomial of degree 4 we have $L=[A]$ and, by Theorem $\ref{generated-by-single-matrix}$, we have that $L$ is isomorphic to one of the spaces $\mathcal L_4$, $\mathcal L_{3,1}$, $\mathcal L_{2,2}$ and $\mathcal L_{2,1,1}$, $\mathcal L_{1,1,1,1}$.

Let us assume that the generic matrix in $L$ is of type $(1,1,1)$ (i.e., its minimal polynomial has three distinct roots).

Consider any two matrices $X,Y\in L$ of type $(1,1,1)$ and $[X]\neq[Y]$. Then $[X]\cap[Y]$ is a self-inverse space. Since $\dim ([X])=\dim([Y])=3$, we must have $4\ge \dim([X]+[Y])>3$, i.e. $\dim([X]+[Y])=3$. It follows that $\dim([X]\cap[Y])=2$.
Since there are no matrices that have multiple roots of the minimal polynomial in $\mathcal L_{1,1,1}$, we must have $[X]\cap[Y]\cong L_{1,1}$.

Note that $[X]\cong \mathcal L_{1,1,1}$  which is the space of $3\times 3$ diagonal matrices, i.e. $\mathcal L_{1,1,1}=\{\diag(x,y,z)\}$. Moreover $\mathcal L_{1,1,1}\cap\mathcal M_{1,1}$ consist of three planes - $x=y; y=z;x=z$ Clearly, one of these planes is $[X]\cap [Y]$, without loss of generality we may assume it is $y=z$. Let $A,B\in [X]$ be a matrices that correspond to $\diag(1,0,0)$ and $\diag(0,1,0)$.  Clearly, $\{\Id, A,B\}$ generates $[X]$ and we have $A^2=A, B^2=B$ $AB+BA=0$.

Analogously, we pick the basis  $\{\Id,A,C\}$ for $[Y]$ with the property that $C^2=C, AC+CA=0$ (it is the same matrix $A$, since $A\in [X]\cap [Y])$.

Clearly, $\{\Id,A,B,C\}$ is a basis for $L$.  Note that $B,C,\Id-A\in Z_A$ and $A\in F_A$. Since $F_A\cap Z_A=0$, it means that $B,C,\Id-A$ for a basis of $Z_A$ and $A$ is a basis of $F_A$ and $F_A\oplus Z_A=L$. By Theorem \ref{splitting-by-I} we have $L\cong F_A\times Z_A\cong \mathcal L_1\times Z_A$, where $Z_A$ is a self-inverse space of dimension 3. 

Since we know all six self-inverse spaces of dimension 3, we can see that, in this case, $L$ is isomorphic to one of the six products. Actually, we can even argue that in this case (generic element is of type $(1,1,1)$) we must have that $Z_A$ ie, isomorphic to either all symmetric $2\times 2$ matrices or all upper triangular $2\times 2$ matrices, but it is not necessary.

\textbf{General matrix in $L$ is of type $(2,1)$:}

Consider any two matrices $X,Y\in L$ of type $(2,1)$ and $[X]\neq[Y]$. Then $[X]\cap[Y]$ is a self-inverse space. Analogously, as in the previous case we have $\dim ([X])=\dim([Y])=3$ and $\dim (L)=4$. It follows that $\dim([X]\cap[Y])=2$.

Thus $[X]\cap [Y]$ is isomorphic to either $\mathcal L_2$ or $\mathcal L_{1,1}$. We will distinguish these two cases. We start with the case $\mathcal L_{1,1}$.

Note that the intersection $\mathcal M_{1,1}\cap \mathcal L_{2,1}$ is formed by the diagonal matrices from $\mathcal L_{2,1}$. Thus, let $A\in [X]$ be a matrix that corresponds to $\diag (0,0,1)$ and let $B\in [X]$ be a matrix that corresponds to the matrix $\begin{pmatrix}
 0&1&0\\
 0&0&0\\
 0&0&0
\end{pmatrix}$. 

Clearly $\{A,B,\Id\}$ form a basis of $[X]$ and we have $A^2=A, B^2=0,A*B=0$. Analogously, we find a matrix $C$ such that $\{A,C,\Id\}$ is a basis of $[Y]$ and we have $C^2=0$ and $A*C=0$.

We note that $B,C,\Id-A\in Z_A$, and, analogously as in the previous case, we can conclude that $L=Z_A\oplus F_A$. By Theorem \ref{splitting-by-I}, we have $L\cong \mathcal L_1\times Z_A$, where $Z_A$ is a three-dimensional self-inverse space. Since a general matrix in $L$ is of type $(2,1)$, a general matrix in $Z_A$ is of type $(2)$, thus $Z_A\cong \Gamma_3$.

We return to the case $([X]\cap[Y])\cong \mathcal L_{2}$. We note that this must hold for any matrices $X,Y$, otherwise the situation was already discussed in the previous case. We note that the intersection $\mathcal M_{2}\cap \mathcal L_{2,1}$ is 2-dimensional subspace given by matrices that have all three diagonal elements equal. 

Let $A\in[X]$ be a matrix that corresponds to $\begin{pmatrix}
 0&1&0\\
 0&0&0\\
 0&0&0
\end{pmatrix}$ and $B\in[X]$ be a matrix that corresponds to $\diag(0,0,1)$. Then $\{A,B,\Id\}$ is a basis of $[X]$, $\{A,\Id\}$ is a basis of $[X]\cap [Y]$ and we have $A^2=A*B=0, B^2=B$.

Analogously, we pick an element $C\in[Y]$ such that $\{A,C,\Id\}$ is a basis of $[Y]$ and $A*C=0, C^2=C$.

Let us consider a matrix $B+kC$, for any $k\in \CC\setminus{0}$. Consider the space $[B+kC]$.  Suppose that this space is 3-dimensional. Since a generic matrix in $L$ is of a type $(2,1)$ it follows that $B+kC$ is of a type $(2,1)$ or $(3)$. Let us consider the intersection $[B+kC]\cap[X]$. Clearly, it is a 2-dimensional space and we must have $[B+kC]\cap [X]\cong \mathcal L_2$.  Thus $[B+kC]\cap [X]=[A]$.

It follows that $A\in [B+kC]$ and $(B+kC)^2\in \spann(A,\Id, B+kC)$. Note that this is true even in the case where $b+kC$ has the minimal polynomial of degree 2.

Let us write product $B*C$ as a linear combination: $B*C=aA+bB+cC+d\Id$. 
Then $(B+kC)^2=B+k^2C+2k\cdot(aA+bB+cC+d\Id)$. Thus $(1+2kb)B+(k^2+2kc)C\in \spann(B+kC)$. This is equivalent to $k(1+2kb)=k^2+2kc$. Since this must be true for all $k\neq 0$, it must be a polynomial identity, i.e., $b=c=1/2$.

Let us consider the decomposition of the vector space $L$ as $L=F_B\oplus Z_B\oplus H_B$ as in Theorem \ref{FZ}. We have $B\in F_B$ and $\Id-B, C\in Z_B$. Clearly $H_B\neq 0$, otherwise the space $L$ decomposes, and this situation was already handled in the previous cases. (One can even show that $H_B=0$ would contradict $B*C=aA+B/2+C/2+dI)$ but it is not necessary. It follows that $\dim(H_B)=\dim (F_B)=1, \dim(Z_B)=2$. Let $H$ be a generator of $H_B$. By Theorem \ref{FZ}, $H*A\in H_B$. However, we can easily see that $A*L\subseteq \spann(A)$. It follows that $H*A=0$. Therefore $H\in\spann(A,B,C)$. We have $B*H=H/2\in\spann(A,B,C)$. Since $A,B,\Id,H$ spans $L$ we have $B*L\subseteq \spann(B*A,B*B,B*\Id,B*H)\subseteq\spann(A,B,C)$. In particular, we have $B*C\in\spann (A,B,C)$ or, in other words, $d=0$.

We have two options $a=0$ or $a\neq 0$. In the case $a\neq 0$, we can replace $A$ by a suitable multiple of $A$ to obtain the same situation with $a=1$. This uniquely (up to isomorphism) determines the space $L$. We will provide examples of such spaces.

Let $A=\begin{pmatrix}
 0&1&0\\
 0&0&0\\
 0&0&0
\end{pmatrix}, 
B=\begin{pmatrix}
 0&0&1\\
 0&0&0\\
 0&0&1
\end{pmatrix},
C=\begin{pmatrix}
 0&0&0\\
 0&0&0\\
 0&0&1
\end{pmatrix}.$

We can see that these matrices satisfy all the relations and $B*C=B/2+C/2$. These matrices generate the space $\Gamma_{3,1}$.

For the other example, let $A=\begin{pmatrix}
 0&0&-1/2\\
 0&0&0\\
 0&0&0
\end{pmatrix}, 
B=\begin{pmatrix}
 0&0&0\\
 0&1&0\\
 0&0&0
\end{pmatrix},
C=\begin{pmatrix}
 0&1&1\\
 0&1&1\\
 0&0&0
\end{pmatrix}.$

In this case, we have $B*C=A+B/2+C/2$, i.e., $a=1$. These matrices generate the space $\Gamma_{2,2}$.

Note that $\Gamma_{3,1}\cap \mathcal M_{1,1}$ is given by the equation $b(a-d)=0$,i.e. it consists of 2 planes. On the other hand, $\Gamma_{2,2}\cap \mathcal M_{1,1}$ is given by the equation $c^2+ad-bd$, i.e., it is a quadric. Thus $\Gamma_{3,1}\not\cong \Gamma_{2,2}$.

\textbf{General matrix in $L$ is of type $(3)$:}

First, we notice that $\mathcal L_3\cap\mathcal M_{1,1}=\spann(\Id)$ and $\mathcal L_3\cap \mathcal M_2$ is formed by the plane where elements on the second diagonal are equal to 0.

Now we consider two matrices $X,Y\in L$ of type $(3)$ with $[X]\neq [Y]$. Then $[X]\cap [Y]$ is a 2-dimensional self-inverse space, and it must be isomorphic to $\mathcal L_2$, since the space $[X]\cong\mathcal L_3$ does not contain a subspace isomorphic to $\mathcal L_{1,1}$.

Let $A$ be a matrix from $[X]$ that corresponds to $\begin{pmatrix} 0&0&1\\
 0&0&0\\
 0&0&0
\end{pmatrix}$, and let $B$ be a matrix from $X$ that corresponds to $\begin{pmatrix}0&1&0\\
 0&0&1\\
 0&0&0
\end{pmatrix}$. We have $A^2=0, B^2=A, A*B=0$. Moreover, $[X]\cap [Y]=[A]$. Then we, analogously, find an element $C\in[Y]$ such that $C^2=A$, $A*C=0$ and $\{A,C,\Id\}$ is a basis of $[Y]$.

If we consider the element $B+kC$ for any $k\in\CC\setminus\{0\}$, then, as in the previous case, we must have that $(B+kC)^2\in\spann(A,\Id,B+kC)$. If we denote $B*C=aA+bB+cC+d\Id$, we obtain $(B+kC)^2=A+k^2A+2k(aA+bB+cC+d\Id)$. It follows that $bB+cC\in\spann(B+kC)$ for all $k\in\CC\setminus\{0\}$, i.e. $b=c=0$.  Let us pick $k_0$ as a root of the quadratic equation $1+k^2+2ka=0$. Then $(B+k_0C)^2=2k_0d\Id$. It follows that the minimal polynomial for $B+k_0C$ is $x^2-2k_0d$. However, since $L\subseteq \mathcal M_3$, there are no elements of type $(1,1)$ in $L$. It follows $2k_0d=0$, i.e $d=0$.

Let us assume $a\neq \pm1$. Then the quadratic equation $k^2+2ka+1=0$ has two different roots $k_0$ and $k_1$ and $\{A,\Id, B+k_0C,B+k_1C\}$ for a basis of $L$. Moreover, we have $A^2=A*(B+k_0C)=A*(B+k_1C)=(B+k_0C)^2=(B+k_1C)^2=0$ and $(B+k_0C)*(B+k_1C)=
A+(k_0+k_1)aA+k_0k_1A=(1-2a^2+1)A=2(1-a^2)*A$.

If we denote $B'=B+k_0C$ and $C'=(B+k_1C)/(2-2a^2)$, we have a basis $\{A,B',C',\Id\}$ of $L$ with the property $A^2=A*B'=A*C'=(B')^2=(C')^2=0, B'*C'=A$ which determines the space $L$ up to an isomorphism. Thus, we see that for any value $a\neq \pm 1$ we always get an isomorphic space.

We will provide an example in the original basis $\{A,B,C,\Id\}$ with $a=0$:

$$A=\begin{pmatrix}0&0&0&1\\
 0&0&0&0\\
 0&0&0&0\\
 0&0&0&0
\end{pmatrix}, 
B=\begin{pmatrix}0&1&0&0\\
 0&0&0&1\\
 0&0&0&0\\
 0&0&0&0
\end{pmatrix},
C=\begin{pmatrix}0&0&1&0\\
 0&0&0&0\\
 0&0&0&1\\
 0&0&0&0
\end{pmatrix}.$$

These matrices generate the space $\Delta_4$.

If $a=-1$, we can simply replace $B$ with $-B$ to get to the case $a=1$, so we can assume $a=1$. This, again, determines the space uniquely. We will provide an example of matrices $A,B,C$ that generate such a space:

$$A=\begin{pmatrix}0&0&1&0\\
 0&0&0&0\\
 0&0&0&0\\
 0&0&0&0
\end{pmatrix}, 
B=\begin{pmatrix}0&1&0&0\\
 0&0&1&0\\
 0&0&0&0\\
 0&0&0&0
\end{pmatrix},
C=\begin{pmatrix}0&1&0&1\\
 0&0&1&0\\
 0&0&0&0\\
 0&0&0&0
\end{pmatrix}.$$

These matrices generate the space $\Omega_4$.

Note that the intersection $\Delta_4\cap\mathcal M_2$ is given by the equation $b^2+c^2=0$, i.e., it is the union of two (complex) planes. On the other hand, the intersection $\Omega_4\cap \mathcal M_2$ is given by the equation $b=0$, i.e., it is a plane. Therefore, $\Omega_4\not\cong \Delta_4$.

\textbf{General matrix in $L$ is of type $(1,1)$ or $(2)$:}

This case was already discussed in Section~\ref{classification_type}. There are three 4-dimensional spaces of type $(1,1)$, namely $\mathcal C_{3,0}\cong \CC^{2\times 2}$,
$\mathcal C_{2,1}\cong\Psi_4$, $\mathcal C_{1,2}\cong\Omega_{2,2}$ and there is one space of type $(2)$ and that is $\Gamma_4$.

\end{proof}
\begin{Remark}
One can again directly check that the spaces $\CC^{2\times 2}, \Psi_4$ and $\Omega_{2,2}$ are non-isomorphic by looking at the set of matrices whose polynomial has a double root.

In the space $\CC^{2\times 2}$, this set is a smooth quadric. In the space $\Psi_4$, it is a union of two planes $b=0$ and $c=0$. In the space $\Omega_{2,2}$, it is a plane given by $a=b$. 
\end{Remark}

\section{Future directions}

We classified all self-inverse spaces up to dimension four. It is clear that if we increase the dimension, a similar analysis will get much more complicated. Maybe it can be simplified, or at least one can find an estimate on the number of different self-inverse spaces. We propose the following questions:

\begin{Question}
Classify all the different self-inverse spaces in higher dimensions.
\end{Question}

\begin{Question}
Let us fix the dimension $d$. Is it true that there are only finitely many $d$-dimensional self-inverse spaces (up to isomorphism)?
\end{Question}
\begin{Question}
If the answer to the previous question is positive, how many $d$-dimensional self-inverse spaces are there? If the precise answer is too difficult to obtain, can one find at least an asymptotic behaviour of the number of self-inverse spaces?
\end{Question}

\vspace{.1in}
\section*{Acknowledgement}
RD was supported by a Return Humboldt fellowship and by a grant of the Ministry of Education and Research, CNCS-UEFISCDI, project number PN-IV-P2-2.1-BSM-2024-0022, within PNCDI IV.

\end{document}